\documentclass[letterpaper, 10 pt, conference]{ieeeconf}

\IEEEoverridecommandlockouts	% This command is only
\usepackage{cite}
\usepackage{amsmath,amssymb,amsfonts}
\usepackage{algorithmic}
\usepackage{textcomp}
\usepackage{graphicx,color}
\usepackage{mathrsfs}
\usepackage[vlined,ruled]{algorithm2e}
\usepackage{subfigure}
\usepackage{url}
\usepackage{color}
\usepackage{dsfont}
\usepackage{bbm}
\usepackage{booktabs}
\usepackage{array}
\usepackage[table]{xcolor} 
\usepackage{yfonts}
\usepackage[normalem]{ulem}
\usepackage{arydshln}

\usepackage[colorlinks]{hyperref}

\newtheorem{theorem}{Theorem}[section]
\newtheorem{lemma}[theorem]{Lemma}

\newtheorem{corollary}[theorem]{Corollary}

\newtheorem{remark}{Remark}

\newtheorem{example}{Example}

\newcommand{\subscr}[2]{{#1}_{\textup{#2}}}

\usepackage{tabularx}
\usepackage{multirow}
\usepackage{makecell}

\newcommand\aamsout{\bgroup\markoverwith{\textcolor{violet}{\rule[0.5ex]{2pt}{1pt}}}\ULon}

\newcommand{\real}{\mathbb{R}}

\newcommand{\transpose}{\mathsf{T}} %or \top or \intercal

\newcommand{\mc}{\mathcal}

\newcommand{\Trace}{\mathsf{Tr}}

\DeclareSymbolFont{bbold}{U}{bbold}{m}{n}
\DeclareSymbolFontAlphabet{\mathbbold}{bbold}

\newcommand\oprocendsymbol{\hbox{$\square$}}
\newcommand\oprocend{\relax\ifmmode\else\unskip\hfill\fi\oprocendsymbol}

\renewcommand{\theenumi}{(\roman{enumi})}

\newcommand*{\QEDA}{\hfill\ensuremath{\blacksquare}}%
\newenvironment{pfof}[1]{\vspace{1ex}\noindent{\itshape Proof of
    #1:}\hspace{0.5em}} {\hfill\QEDA\vspace{1ex}}

\graphicspath{{figs/}}

\renewcommand{\baselinestretch}{0.98}

\begin{document}

\title{\bf Accuracy Prevents Robustness in Perception-based Control}
%\title{\bf \textcolor{violet}{Accuracy Prevents Robustness to Changes of Data Noise Statistics in Linear Estimators}}

% \title{\bf Performance Limitations of Perception-Based Estimation in
%   Non-Nominal Environments}

\author{Abed~AlRahman~Al~Makdah,~Vaibhav~Katewa,~and~Fabio~Pasqualetti%
  \thanks{This work was supported in part by ARO award 71603NSYIP and
    in part by ONR award N00014-19-1-2264. The authors are with the
    Department of Electrical and Computer Engineering and the
    Department of Mechanical Engineering at the University of
    California, Riverside,
    \href{mailto:aalm005@ucr.edu}{\{\texttt{aalmakdah}},\href{mailto:vkatewa@engr.ucr.edu}{\texttt{vkatewa}},\href{mailto:fabiopas@engr.ucr.edu}{\texttt{fabiopas\}@engr.ucr.edu}}.}}

\maketitle

\begin{abstract}
  In this paper we prove the existence of a fundamental trade-off
  between accuracy and robustness in perception-based control, where
  control decisions rely solely on data-driven, and often incompletely
  trained, perception maps. In particular, we consider a control
  problem where the state of the system is estimated from measurements
  extracted from a high-dimensional sensor, such as a camera. We
  assume that a map between the camera's readings and the state of the
  system has been learned from a set of training data of finite size,
  from which the noise statistics are also estimated. We show that
  algorithms that maximize the estimation accuracy (as measured by the
  mean squared error) using the learned perception map tend to perform
  poorly in practice, where the sensor's statistics often differ from
  the learned ones. Conversely, increasing the variability and size of
  the training data leads to robust performance, however limiting the
  estimation accuracy, and thus the control performance, in nominal
  conditions. Ultimately, our work proves the existence and the
  implications of a fundamental trade-off between accuracy and
  robustness in perception-based control, which, more generally,
  affects a large class of machine learning and data-driven algorithms
  \cite{AALM-VK-FP:19,DT-SS-LE-AT-AM:19,ZD-CD-JW-YZ:19,HZ-YY-JJ-EPX-LEG-MIJ:19b}.
\end{abstract}

%\begin{IEEEkeywords}
%
%\end{IEEEkeywords}

\section{Introduction}\label{sec: introduction}
Machine learning methods are rapidly being deployed for a broad class
of applications, ranging from speech recognition and malware
detection, to control design and dynamic decision making. These
data-driven algorithms often outperform classical methods and require,
typically, substantially less knowledge about the specifics of the
problem. For control applications, in particular, data-driven
algorithms promise to overcome the limitations of traditional
model-based approaches, and to provide solutions to complex control
problems where a detailed model of the plant and its operating
environment is either too complex to be useful, or too difficult to
estimate or derive from first principles
\cite{BR:18,FLL-DV-KGV:12,PZ-JI-BF-SF:17}. Yet, the lack of
strong guarantees for the safety and robustness of data-driven
algorithms questions their deployment, especially in applications such
as autonomous driving and exploration. 

In this paper, we characterize a fundamental trade-off between
accuracy and robustness in a data-driven control problem. We consider
a perception-based control scenario, Fig. \ref{fig: perc_cont}, where a camera is used to
partially measure the state of a dynamical system and construct an
estimator of the full state. We assume that the output map between the
high-dimensional camera stream and the system state has been learned
accurately \cite{SD-NM-BR-VY:19}, although the estimated statistics of
the measurement noise are inaccurate. Such inaccuracies, which can
arise from limited training data, sudden changes in environmental
conditions, and adversarial manipulation, are unknown to the estimator
and induce incorrect confidence bounds on the estimated state
variables. In turn, inaccurate confidence bounds can lead to harmful
control decisions \cite{SL:16}. Further, we show that, because of the
incorrect noise statistics, accuracy of the estimation algorithm can
be improved only at the expenses of its robustness. Thus, estimation
algorithms that are optimal in the nominal training phase may
underperform in practice compared to suboptimal algorithms. Our
analytical results provide an explanation as to why nominally
suboptimal data-driven algorithms can exhibit better generalization
and robust properties in practice~\cite{AI-SS-DT-LE-BT-AM:19b}.

% \textcolor{violet}{Estimation algorithms, in particular, fall under
%   the umbrella of data-driven algorithms, where they are designed to
%   provide an estimate with some covariance based on measurement
%   data. The covariance of an estimation algorithm depends on the
%   measurement noise statistics. When an estimator is designed based on
%   nominal measurement noise statistics, and is deployed in a
%   non-nominal environment with different noise statistics, the
%   estimation error increases. However, in this scenario, the estimator
%   will assume that this error still lies within the nominal
%   covariance, which may lead to catastrophic control decisions
%   \cite{SL:16}.}

\begin{figure}[!t]
  \centering
  \includegraphics[width=\columnwidth]{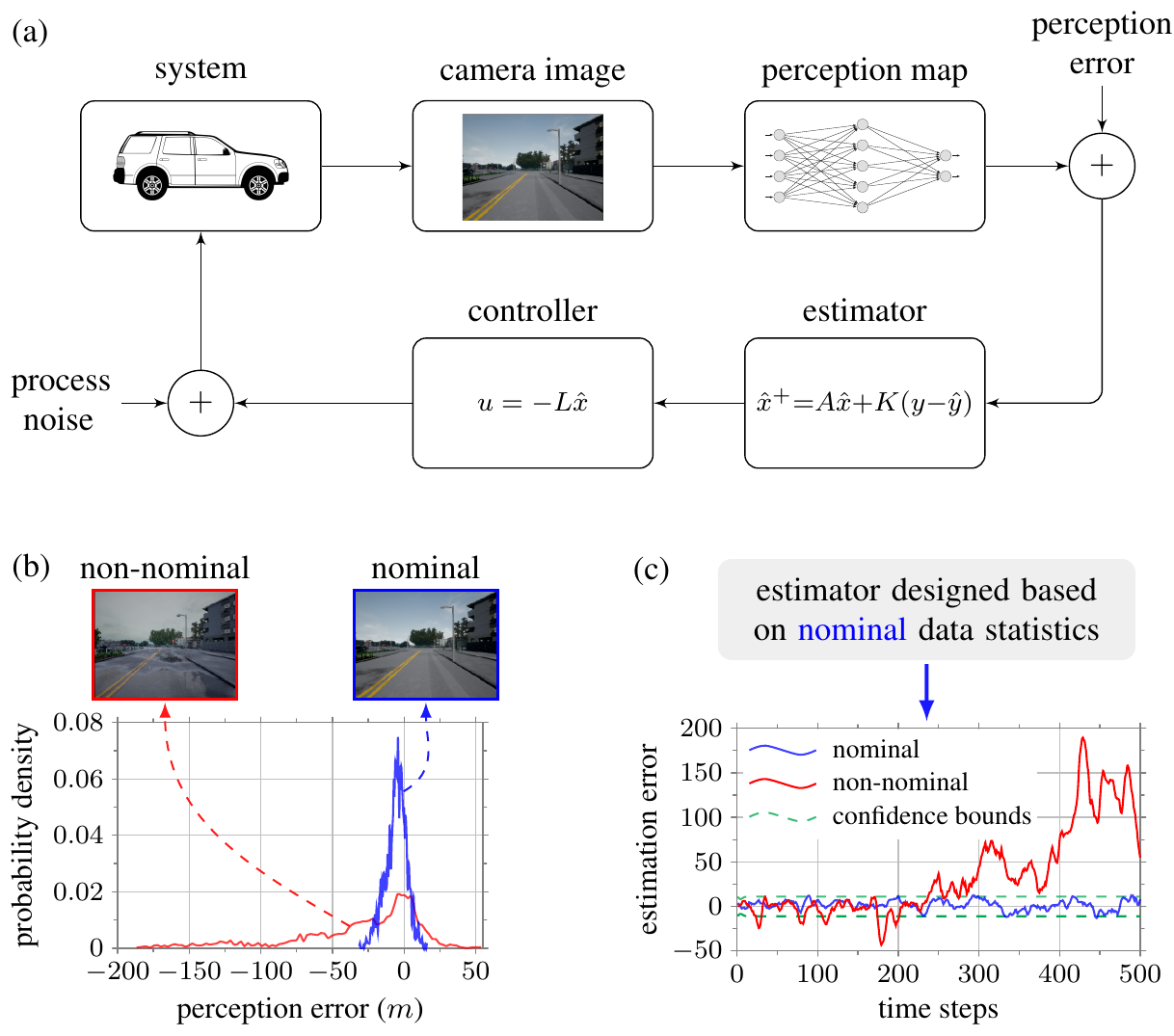}
  \caption{Panel (a) shows a perception-based control scenario, where
    the partial state of a dynamical system (vehicle) is extracted
    from the measurements of a high-dimensional sensor (camera) and
    used to implement a feedback control algorithm. A perception map
    is learned from a set of training data of finite size, which
    relates the sensor's readings to the system's state. Panel (b)
    shows the probability density functions of the perception error
    when operating in nominal (clear weather, as represented by the
    training data) and non-nominal (rainy weather, as it may occur in
    practice) conditions (error statistics are computed numerically
    using the simulator CARLA \cite{AD-GR-FC-AL-VK:17}). Due to
    inaccuracies and uncertainties in the sensed data, the error
    statistics of the perception map differ from the statistics
    learned during the training phase. As shown in panel (c),
    discrepancies in the error statistics lead to poor estimation
    performance in practical conditions. As we prove in this paper, a
    fundamental trade-off exists between accuracy and robustness of a
    linear estimator (consequently, in the considered perception-based
    control setting), so that estimators that perform well on the
    training data may exhibit poor performance with non-nominal
    conditions, while robust estimators may exhibit mediocre yet
    robust performance in a broad set of conditions.}
  \label{fig: perc_cont}
\end{figure}

% \textcolor{violet}{In this paper, we characterize a fundamental trade-off between the accuracy of a linear estimator and its robustness to the variation in the measurement noise statistics. We consider a scenario where we have a closed loop system with a linear estimator deployed in a non-nominal environment, where noise statistics is different than the ones used for the estimator design.} Such \textcolor{violet}{difference} \aamsout{inaccuracies} can arise,
% for instance, from limited \aamsout{training} data \textcolor{violet}{to learn the statistics}, sudden changes in
% environmental conditions, and adversarial manipulation. We show that,
% because of the inaccuracies in the noise statistics, accuracy of the
% estimation algorithm can be improved only at the expenses of its
% robustness.  Thus, estimation algorithms that are optimal in \aamsout{the
% training phase} \textcolor{violet}{nominal environments} may underperform in practice compared to suboptimal
% algorithms, which provides an \aamsout{explanation} \textcolor{violet}{intution} as to why suboptimal
% learning algorithms may exhibit better generalization
% properties~\cite{AI-SS-DT-LE-BT-AM:19b}.

%\cite{IA-YB-AC:16}

\smallskip
\noindent
\textbf{Related work.} Machine learning and, more generally,
data-driven algorithms have shown remarkable performance under nominal
and well-modeled conditions in a variety of applications. Yet, the
same algorithms have proven extremely fragile when subject to small,
yet targeted, perturbations of the data \cite{CS-WZ-IS-JB-DE-IG-RF:14,
  IJG-JS-CS:14}. A detailed understanding of this unreliable behavior
is still lacking, with recent theoretical results proving robustness
and generalization guarantees for learning algorithms subject to
adversarial disturbances, e.g., see
\cite{SY-KP:18,BH-TK:01,OA-EH-SM:15}, and showing that, in certain
contexts, robustness to perturbations and performance under nominal
conditions are inversely related
\cite{AALM-VK-FP:19,DT-SS-LE-AT-AM:19,ZD-CD-JW-YZ:19,HZ-YY-JJ-EPX-LEG-MIJ:19b}.
Compared to these works, we prove that a fundamental trade-off between
accuracy and robustness also arises in linear estimation algorithms,
which may lead to a critical degradation of the closed loop
performance \cite{SL:16}.

Related to this work is the literature on robust control and
estimation \cite{KZ-JCD:98,NM-LM:96}. However, the primary focus of
this paper is not on designing a robust estimator or controller, but
rather on proving the existence of a fundamental trade-off between
accuracy and robustness, which plays a critical role in the deployment
of learning and data-driven methods in control applications, including
perception-based control.

Finally, the literature on perception-based control is also very rich,
with results ranging from integrating camera measurements with
inertial odometry \cite{JK-GSS:11}, to control of unmanned aerial
vehicles \cite{GL-CB-GM-VK:16} and vision-based planning
\cite{SB-VT-SG-JM-CT:19b}, to name a few. To the best of our knowledge,
the trade-off between accuracy and robustness that we highlight here
was not discussed in any of the above research streams.

% \textcolor{violet}{In recent work, deep learning algorithms are used
% to estimate the states of a dynamical system using a perception map
% (neural network), which maps the output of a high dimensional
% sensor, such as a camera, to the system's states
% \cite{AL-AS-AR-ZL-BB:18 , IA-YB-AC:16}. At the end of this paper, we
% apply our results in a perception-based control scenario. Where a
% camera is used to partially measure the state of a dynamical system
% via a trained perception map, and reconstruct the system's state
% using an estimator.}

\smallskip
\noindent
\textbf{Paper contributions.} This paper features two main
contributions. First, we study a perception-based control problem,
where the state of a dynamical system is reconstructed using a
high-dimensional sensor. We prove the existence of a fundamental
trade-off between the accuracy of the estimation algorithm, as
measured by its minimum mean squared error, and its robustness to
variations and inaccuracies of the data statistics. Thus, (i)
estimation algorithms that are optimal for the nominal data tend to
perform poorly in practice, where the operating conditions may differ
from the nominal data, and, conversely, (ii) estimation algorithms
that are robust to data variations exhibit suboptimal performance in
nominal conditions. Second, we characterize estimators that lie on the
Pareto frontier between accuracy and robustness, that is, estimators
that are maximally robust for a desired performance level, and
estimators that are maximally accurate for a given bound on the data
variations and inaccuracies. We also show, numerically, that the
trade-off for estimation algorithms also affects the performance of
the closed-loop~system, and even when the measurement error is not
normally distributed, as we assume for the derivation of our
analytical results.

In a broader context, the results of this paper further characterize a
fundamental limitation of machine learning and data-driven algorithms,
as described for different settings in
\cite{AALM-VK-FP:19,DT-SS-LE-AT-AM:19,ZD-CD-JW-YZ:19,HZ-YY-JJ-EPX-LEG-MIJ:19b},
and clarify its implications for control applications.

\smallskip
\noindent
\textbf{Paper's organization.} The rest of the paper is organized as
follows. Section \ref{sec: setup} contains our mathematical
setup. Section \ref{sec: trade-off} contains the trade-off between
accuracy and robustness, and the design of optimal estimators. Section
\ref{sec: example} contains our numerical example, and Section
\ref{sec: conclusion} concludes the paper.

\smallskip
\noindent
\textbf{Notation.} A Gaussian random variable $x$ with mean $\mu$ and
covariance $\Sigma$ is denoted as $x\sim\mc{N}(\mu,\Sigma)$. The
$n\times n$ identity matrix is denoted by $I_n$. The expectation
operator is denoted by $\mathbb{E}[\cdot]$. The spectral radius and
the trace of a square matrix $A$ are denoted by $\rho(A)$ and
$\Trace(A)$, respectively. A positive definite (semidefinite) matrix
$A$ is denoted as $A>0$ ($A\geq 0$). The Kronecker product is denoted
by $\otimes$, and vectorization operator is denoted by vec($\cdot$).

% \aamsout{We focus on the perception-based control scenario
% illustrated in Fig.  \ref{fig: perc_cont}, where measurements of the
% partial state of a dynamical system are extracted from a
% high-dimensional sensor (camera) through a perception map. We assume
% the perception map to be learned offline using a set of training
% data of finite size. Using the perception map and the noise
% statistics estimated from the training data, the state of the system
% is reconstructed via an estimator, and ultimately used for control
% purposes. As shown in Fig. \ref{fig: perc_cont} and articulated in
% this paper, inaccuracies in the perception map and the noise
% statistics may lead to unexpected results that (i) optimal
% estimators, which perform well on the training data, may exhibit
% poor performance when deployed in practice, and (ii) robust
% estimators, which are obtained using a larger and more diverse
% training set, may exhibit mediocre performance in nominal
% conditions. Thus, a fundamental trade-off relates estimation
% accuracy -- hence performance of perception-based control -- and
% robustness to data and model inaccuracies. The details of the
% perception-based model are in Section \ref{sec: example}.}

\section{Problem setup and preliminary notions}\label{sec: setup}
Consider the discrete-time, linear, time-invariant system
\begin{align}
  x(t+1) &= Ax(t) + w(t), \label{eq: dynamical model} \\ 
  y(t) &= Cx(t) + v(t), \qquad t\geq 0,\label{eq: output model}
\end{align}
where $x(t)\in\real^{n}$ denotes the state, $y(t)\in\real^{m}$ the
output, $w(t)$ the process noise, and $v(t)$ the measurement noise. We
assume that $w(t)\sim\mc{N}(0,Q)$, with $Q\geq 0$,
$v(t)\sim\mc{N}(0,R)$, with $R>0$, and $x(0)\sim\mc{N}(0,\Sigma_0)$,
with $\Sigma_0 \ge 0$, are independent of each other at all times
$t\geq 0$.\footnote{See Section \ref{sec: example} for numerical
  examples showing that our main results seem to be valid also when
  some of these assumptions are not satisfied.} Finally, we assume
that $A$ is stable, that is, $\rho(A)<1$. Note that this implies that
$(A,C)$ is detectable and $(A,Q^{\frac{1}{2}})$ is stabilizable.

% The initial condition $x(0)\sim\mc{N}(0,\Sigma_0)$ is assumed to be
% independent of $w(t)$ and $v(t)$ for all $t\geq 0$.

We use a linear filter with constant gain $K\in\real^{n\times m}$ to
estimate the state of the system \eqref{eq: dynamical model} from the
measurements~\eqref{eq: output model}:
\begin{align}\label{eq:lin_filter}
  \hat{x}(t+1) = A\hat{x}(t) + K[y(t+1)-CA\hat{x}(t)] \quad t\geq 0,
\end{align} 
where $\hat{x}(t)$ denotes the state estimate at time $t$. Let
$e(t)=x(t)-\hat{x}(t)$ and $P(t) = \mathbb{E}[e(t)e(t)^{\transpose}]$
denote the estimation error and its covariance, respectively. For
$t\geq 0$, we have
\begin{align}
e(t+1) &= A_K e(t) + B_K w(t) - Kv(t+1),\\
P(t+1) &= A_K P(t)A_K^{\transpose} + B_K Q B_K^{\transpose} + K R K^{\transpose},
\end{align}
where $A_K \triangleq A-KCA$ and $B_K \triangleq I_n-KC$. We assume
that the gain $K$ is chosen such that $A_K$ is stable, that is,
$\rho(A_K)<1$. Under this assumption,
$\underset{t\rightarrow \infty }{\lim} P(t) \triangleq P(K) \geq 0 $
exists, and satisfies the Lyapunov equation
\begin{align}\label{eq: lyapunov equation error covariance}
P(K)= A_KP(K)A_K^{\transpose} + B_K Q B_K^{\transpose} + K R K^{\transpose}.
\end{align}
The performance of the filter is quantified by
$\mathcal{P}(K) \triangleq \Trace(P(K))$, where a lower value of
$\mc{P}(K)$ is desirable. Note that the steady-state gain
$\subscr{K}{kf}$
%\margin{Abed, please use $\subscr{K}{kf}$ everywhere  in the paper} 
  of the Kalman filter \cite{TK:80} minimizes
$\mc{P}(K)$ and depends on the matrices $A$, $C$, $Q$, $R$.

We allow for perturbations to the covariance matrix $R$, which may
result from (i) modeling and estimation errors, as in the case of
perception-based control, or (ii) accidental or adversarial tampering
of the sensor, as in the case of false data injection attacks
\cite{FP-FD-FB:10y}. To quantify the effect of such perturbations to
the covariance matrix $R$ on the performance of the estimator, we
define the following sensitivity metric:
\begin{align}\label{eq: definition of sensitivity}
\mc{S}(K) \triangleq \Trace \left[\frac{d}{dR}\mc{P}(K)\right].
\end{align}
Intuitively, if $\mc{S}(K)$ is large, then a small change in $R$ can
result in a large change (possibly, large increment) in~$\mc{P}(K)$.

\begin{remark}{\bf \emph{(Comparison with adversarial
      robustness)}}\label{remark: adversarial robustness}
  In adversarial settings, the adversary designs a small deterministic
  perturbation added to a given observation (e.g., pixels of an image)
  to deteriorate the performance of a machine learning algorithm. This
  perturbed observation can be viewed as a realization of a
  multi-dimensional distribution. Instead, in this work we consider
  perturbations to the sensor's noise covariance, which accounts for
  all possible realizations. Thus, our sensitivity metric captures the
  average performance change over all possible perturbations, rather
  than the degradation caused by a single worst-case
  perturbation.~\oprocend
\end{remark}

Lower values of sensitivity $\mc{S}(K)$ are desirable, and indicate
that the filter \eqref{eq:lin_filter} is more robust to
perturbations. This motivates the following optimization problem:
\begin{equation}\label{eq:Opt_prob}
\begin{aligned}
\mc{S}^{*}(\delta) = & \;\underset{K}{\min} \quad \mc{S}(K)\\
&\;\text{s.t.} \quad  \mc{P}(K) \leq \delta,
\end{aligned}
\end{equation}
where $\delta \geq \mc{P}(\subscr{K}{kf})$ for feasibility. In what
follows, we characterize the solution $K^*$ to \eqref{eq:Opt_prob},
and the relations between the sensitivity $\mc S(K^*)$ and the error
$\mc P(K^*)$ as $\delta$ varies. To facilitate the
discussion, in the remainder of the paper we use \emph{accuracy} to
refer to any decreasing function of the error $\mc P(K)$ obtained by
the gain $K$, and \emph{robustness} to denote any decreasing function
of the sensitivity $\mc S(K)$ of the gain $K$.

% We aim to quantify the trade-off between the performance and
% sensitivity of the estimator. Mathematically, we wish to
% characterize the function $\mc{S}^{*}(\delta)$, as $\delta$ varies.

\section{Accuracy vs robustness trade-off in \\linear estimation
  algorithms}\label{sec: trade-off}
We begin by characterizing the sensitivity $\mc{S}(K)$.
\begin{lemma}{\bf \emph{(Characterization of
      sensitivity)}}\label{lem:sens_charac}
  Let the sensitivity $\mc{S}(K)$ be as in \eqref{eq: definition of
    sensitivity}. Then, $\mc{S}(K) = \Trace(S(K))$, where
  $S(K)\geq 0 $ satisfies the following Lyapunov equation:
  \begin{align}\label{eq:S_Lyap}
    S(K) = A_K S(K) A_K^{\transpose} +
    K K^{\transpose}.
  \end{align}
\end{lemma}

\smallskip
Lemma \ref{lem:sens_charac} allows us to compute the sensitivity of
the linear estimator \eqref{eq:lin_filter} as a function of its gain.
Before proving Lemma \ref{lem:sens_charac}, we present the following
technical result.

%{\aammargin{for a unique solution to exist, $Q>0$
%    and $B>0$}}\margin{Please address this comment}

\begin{lemma}{\bf \emph{(Property of the solution to Lyapunov
      equation)}}\label{lem:tech_res}
  Let $A$, $B$, $Q$ be matrices of appropriate dimension with
  $\rho(A)<1$. Let $Y$ satisfy $Y = AYA^{\transpose} + Q$. Then,
  $\Trace(BY) = \Trace(Q^{\transpose}M)$, where $M$ satisfies
  $M = A^{\transpose}MA + B^{\transpose}$.
\end{lemma}
\begin{proof}
  Since $\rho(A)<1$, $Y$ and $M$ can be written as
  \begin{align}\label{eq: solution of lyapunov equation Y}
    Y=\sum_{i=0}^{\infty} A^i Q (A^{\transpose})^i \text{ and }M=\sum_{i=0}^{\infty} A^i B (A^{\transpose})^i.
  \end{align}
  The result follows by pre-multiplying $Y$ and $M$ by $B$ and
  $Q^{\transpose}$ respectively, and using the cyclic property of
  trace.
\end{proof}

\begin{pfof}{Lemma \ref{lem:sens_charac}}
  Taking the differential of
  $\eqref{eq: lyapunov equation error covariance}$ with respect to the
  variable $R$, we get
  \begin{align}
    dP(K) &= A_K dP(K) A_K^{\transpose} + K dR K^{\transpose} \nonumber \\ \label{eq:d_Tr_P}
    \Rightarrow d\Trace (P(K)) &= \Trace(dP(K)) \overset{(a)}{=} \Trace(KdRK^{\transpose}M),
  \end{align}
  where $M>0$ satisfies: $M = A_K^{\transpose} M A_K + I_n$, and $(a)$
  follows from Lemma \ref{lem:tech_res}. From \eqref{eq:d_Tr_P}, we get
  \begin{align} \label{eq:d_dR_P} d \mc{P}(K) =
    \Trace(K^{\transpose}MKdR) \Rightarrow \frac{d}{dR} \mc{P}(K) =
    K^{\transpose}MK.
  \end{align}
  Using \eqref{eq:d_dR_P} and \eqref{eq: definition of sensitivity},
  we have that
  $\mc{S}(K) = \Trace(K^{\transpose}MK) = \Trace(KK^{\transpose}M) =
  \Trace(S(K))$, where $S(K)$ is defined in \eqref{eq:S_Lyap} and the
  last equality follows from Lemma \ref{lem:tech_res}. To conclude,
  the property $S(K)\geq 0$ follows by inspection from
  \eqref{eq:S_Lyap}.
\end{pfof}

%\textcolor{purple}{
%\begin{lemma}\label{lem:min_sens} {\bf \emph{(Stabilizing Gain at Minimum Sensitivity)}}
%Let $E\succeq0$ with arbitrarily small elements such that $(A,E^{\frac{1}{2}})$ is stabilizable.
%\end{lemma}}

Notice that, since $S(K)\geq 0$, $\mc{S}(K)=\Trace(S(K))$ is a valid
norm of $S(K)$ and captures the size of $S(K)$. Further, $S(K)=0$ for
$K=0$, that is, $K=0$ achieves the lowest possible value of
sensitivity. This implies that $\delta$ in the optimization problem
\eqref{eq:Opt_prob} can be restricted to $[\mc{P}(K_{\text{kf}}),\mc{P}(0)]$
to characterize the accuracy-robustness trade-off.

Next, we characterize the optimal solution to \eqref{eq:Opt_prob}. We
will show that, despite not being convex, the minimization problem
\eqref{eq:Opt_prob} exhibits a unique local minimum. This implies that
the local minimum is also the global minimum.

% Notice that, although problem \eqref{eq:Opt_prob} is
% not convex, it exhibits a unique local minimum as we will show
% next. This implies that the local minimum is also the global minimum.

\begin{theorem}{\bf \emph{(Solution to the minimization problem
      \eqref{eq:Opt_prob})}}\label{theorem: optimal gain}
  Let $\delta \in [\mc{P}(K_{\text{kf}}),\mc{P}(0)]$ and $\lambda\geq0$. Let
  $X\geq 0$ be the unique solution to the following Riccati equation:
  \begin{align}\label{eq: riccati equation optimal gain}
    X = A X A^{\transpose} \!-\! AXC^{\transpose} (CXC^{\transpose} \!+\! I_m
    \!+\! \lambda R)^{-1} C X A^{\transpose} \!+\! \lambda Q. 
  \end{align}
  Then, the global minimum of problem \eqref{eq:Opt_prob} is given by
  \begin{align}\label{eq:opt_gain}
    K^*(\lambda)=XC^{\transpose}\Big(CXC^{\transpose} +I_m+\lambda
    R\Big)^{-1},
  \end{align}
  where $\lambda$ is selected such that
  $\mc{P}(K^{*}(\lambda)) \triangleq \mc{P}^{*}(\lambda) = \delta$.
\end{theorem}

\smallskip
\begin{proof}
  \emph{First-order necessary conditions:} We begin by computing the
  derivatives of $\mc{P}(K)$ and $\mc{S}(K)$ with respect to the
  variable $K$. For notational convenience, we denote
  $A_K, B_K,P(K)$ and $S(K)$ by $\bar{A},B,P$ and $S$,
  respectively. Taking the differential of \eqref{eq:S_Lyap}, we get
  \begin{align} \label{eq:dS}
    dS &= \bar{A} dS \bar{A}^{\transpose} - dKCAS\bar{A}^{\transpose} - \bar{A}S(dKCA)^{\transpose} + dKK^{\transpose}\nonumber \\
       &+ KdK^{\transpose} \triangleq \bar{A} dS \bar{A}^{\transpose} + Z\\
    &\Rightarrow d\mc{S}(K)  \overset{(a)}{=} \Trace(dS) \overset{(b)}{=} \Trace(Z^{\transpose}M) \nonumber\\
       &= 2\Trace[(-CAS\bar{A}^{\transpose} + K^{\transpose})MdK] \nonumber \\ \label{eq:d_dK_Sens}
       &\Rightarrow \frac{d}{dK} \mc{S}(K) = 2M(K-\bar{A}SA^{\transpose}C^{\transpose}),
  \end{align}
  where $M>0$ satisfies $M = A_K^{\transpose} M A_K + I_n$, and
  $(a)$ and $(b)$ follow from Lemmas \ref{lem:sens_charac} and
  \ref{lem:tech_res}, respectively. A similar analysis of \eqref{eq:
    lyapunov equation error covariance} yields
  \begin{align} \label{eq:d_dK_Perf}
    \frac{d}{dK}\mc{P}(K) = 2M(KR-\bar{A}PA^{\transpose}C^{\transpose}-BQC^{\transpose}).
  \end{align}
  Define the Lagrange function of problem \eqref{eq:Opt_prob} as
  \begin{align}\label{eq: lagrangian}
    \mc L(K,\lambda)=\mc{S}(K) + \lambda \Big(\mc{P}(K)-\delta\Big),
  \end{align}
  where $\lambda$ is the Karush-Kuhn-Tucker (KKT) multiplier. The
  stationary KKT condition implies
  $\frac{d}{dK} \mc{L}(K,\lambda) = 0$, which using
  \eqref{eq:d_dK_Sens} and \eqref{eq:d_dK_Perf} becomes
  \begin{align} \label{eq:d_dK_L}
    2M[K-\bar{A}SA^{\transpose}C^{\transpose} + \lambda (KR
    -\bar{A}PA^{\transpose}C^{\transpose}- BQC^{\transpose})] = 0.
  \end{align}
  Substituting $\bar{A} = A-KCA$ in the above equation, defining
  $X\triangleq A(S+\lambda P)A^{\transpose} + \lambda Q$, and using
  $M>0$, we obtain \eqref{eq:opt_gain}.  Next, we show that $X$
  satisfies \eqref{eq: riccati equation optimal gain}. From \eqref{eq:
    lyapunov equation error covariance} and \eqref{eq:S_Lyap}:
  \begin{align*}
    &S + \lambda P = \bar{A} (S+\lambda P)\bar{A}^{\transpose} + \lambda BQB^{\transpose} + K (I_m + \lambda R) K^{\transpose} \nonumber \\
    &\Rightarrow X = A(S+\lambda P)A^{\transpose} + \lambda Q \\
    &= A\left[\bar{A} (S+\lambda P)\bar{A}^{\transpose} + \lambda
      BQB^{\transpose} + K (I_m + \lambda R) K^{\transpose}
      \right]A^{\transpose} \\&\;\;\;\; + \lambda Q.
  \end{align*}
  Using $\bar{A} = A-KCA$ and substituting the gain $K$ in
  \eqref{eq:opt_gain} in the above equation, we obtain the Riccati
  equation~\eqref{eq: riccati equation optimal gain}.

  The KKT condition for dual feasibility implies that $\lambda\geq 0$,
  so \eqref{eq: riccati equation optimal gain} has a unique
  stabilizing solution. Further, the KKT condition for complementary
  slackness implies $\lambda [\mc{P}(K^{*}(\lambda)) -\delta]=
  0$. Thus, if $\lambda>0$, then $\mc{P}(K^{*}(\lambda)) = \delta$. If
  $\lambda=0$, then the solution to \eqref{eq: riccati equation
    optimal gain} is $X=0$. This implies that $K^{*}(0) = 0$, which is
  feasible only if $\delta = \mc{P}(0)$. Thus, for any
  $\delta \in [\mc{P}(K_{\text{kf}}),\mc{P}(0)]$, it holds
  $\mc{P}(K^{*}(\lambda)) = \delta$.

  \medskip
  \noindent \emph{Second-order sufficient conditions:} We show that the
  stationary point \eqref{eq:opt_gain} corresponds to a local minimum. We begin by
  computing the second-order differential of $\mathcal{S}(K)$. Taking
  the differential of \eqref{eq:dS} and noting that $d^2K=0$, we get
  \begin{align}
    d^2S &= \bar{A}d^{2}S\bar{A}^{\transpose} -2dKCAdS\bar{A}^{\transpose} - 2\bar{A} dS (dKCA)^{\transpose} \nonumber \\
         &+ 2dK(I_p + CASA^{\transpose}C^{\transpose})dK^{\transpose} \nonumber \triangleq \bar{A}d^{2}S\bar{A}^{\transpose} + Y \\
    \Rightarrow d^2\mc{S}(K) &= \Trace(d^2S) = \Trace(YM) = -4\Trace(dKCAdS\bar{A}^{\transpose}M) \nonumber \\ \label{eq:d2_S}
         &+ 2\Trace(dK(I_p + CASA^{\transpose}C^{\transpose})dK^{\transpose}M).
  \end{align}
  Similar analysis of \eqref{eq: lyapunov equation error covariance}
  yields
  \begin{align} \label{eq:d2_P}
    d^2\mc{P}(K) &= -4\Trace(dKCAdP\bar{A}^{\transpose}M) \\
&+ 2\Trace[dK(R + CAPA^{\transpose}C^{\transpose} + CQC^{\transpose})dK^{\transpose}M].\nonumber
  \end{align}
  Adding \eqref{eq:d2_S} and \eqref{eq:d2_P}, we get
  \begin{align*}
    &d^2{\mc{L}} = -4\Trace(dKCA\underbrace{(dS+\lambda
      dP)}_{\overset{(a)}{=}0.} \bar{A}^{\transpose}M) \\
                % &+ 2\Trace[dK\underbrace{(I_p + \lambda R +
                % CA(S+\lambda P)A^{\transpose}C^{\transpose} +
                % \lambda CQC^{\transpose})}_{\triangleq
                % W}dK^{\transpose}M] \\
                &+ 2\Trace[dK W dK^{\transpose}M] 
                  = \text{vec}^{\transpose}(dK) (2W\otimes M) \text{vec}(dK),
  \end{align*}
  where
  $W \triangleq I_p + \lambda R + CA(S+\lambda
  P)A^{\transpose}C^{\transpose} + \lambda CQC^{\transpose}$, and
  % \begin{align*}
      %       W \triangleq I_p + \lambda R + CA(S+\lambda P)A^{\transpose}C^{\transpose} + \lambda CQC^{\transpose},
      %     \end{align*}
  where $(a)$ holds because $d\mc{L}(K,\lambda) = 0$ at the stationary
  point. The above expression implies that the Hessian of the
  Lagrangian is given by $H = 2W\otimes M$, which is positive-definite
  because $W>0$ and $M>0$. Thus, the considered stationary point
  corresponds to a local minimum.
  
\noindent \emph{Uniqueness of $\lambda$:} Next, we show that for a given $\delta$, the equation $\mc{P}(K^{*}(\lambda))=\delta$ has a unique solution. Note that for a given $\lambda>0$, the optimal gain $K^{*}(\lambda)$ in \eqref{eq:opt_gain} is the unique minimizer of the cost $\mc{C}(K) = \mc{S}(K) + \lambda \mc{P}(K)$. Let $\lambda_2> \lambda_1 >0$. Then, we have
\begin{align*}
 \mc{S}(K^{*}(\lambda_1)) + \lambda_1 \mc{P}(K^{*}(\lambda_1)) <  \mc{S}(K^{*}(\lambda_2)) + \lambda_1 \mc{P}(K^{*}(\lambda_2)), \\
 \mc{S}(K^{*}(\lambda_2)) + \lambda_2 \mc{P}(K^{*}(\lambda_2)) <  \mc{S}(K^{*}(\lambda_1)) + \lambda_2 \mc{P}(K^{*}(\lambda_1)).
\end{align*}
Adding the above two equations, we get
$\mc{P}(K^{*}(\lambda_2))<\mc{P}(K^{*}(\lambda_1))$. Thus,
$\mc{P}(K^{*}(\lambda))$ is a strictly decreasing function of
$\lambda$, and therefore, it is one-to-one.

  To conclude the proof, since the necessary and sufficient
  conditions for a local minimum are satisfied by a unique gain, the
  local minimum is also the global minimum.
\end{proof}

\smallskip
\begin{corollary}{\bf \emph{(Properties of 
      $\mc{P}^{*}(\lambda)$)}}\label{cor:monotonicity} The error
  $\mc{P}^{*}(\lambda)$ defined in Theorem \ref{theorem: optimal gain}
  is a strictly decreasing function of~$\lambda$.
\end{corollary}

\smallskip Theorem \ref{theorem: optimal gain} shows that the optimal
gain can be characterized in terms of a scalar parameter $\lambda$,
which depends on the performance level $\delta$ according to the
relation $\mc{P}^{*}(\lambda)=\delta$. Notice that $\lambda = 0$ if
$\delta = \mc{P}(0)$, and $\lambda$ approaches infinity as $\delta$
approaches $\mc{P}(\subscr{K}{\text{kf}})$. In other words,
$\underset{\lambda\rightarrow \infty}{\lim} K^{*}(\lambda) =
\subscr{K}{kf}$.  Further, Corollary \ref{cor:monotonicity} implies
that for a given $\delta$, the solution of
$\mc{P}^{*}(\lambda) = \delta$ can be found efficiently. For instance,
one can use the bisection algorithm on the interval
$[0,\subscr{\lambda}{max}]$, where
$\mc{P}^{*}( \subscr{\lambda}{max})>\delta$. These results also imply
a fundamental trade-off between performance and robustness of the
estimator.

% Next, we show that a trade-off exists between sensitivity and
% performance of the linear filter.

\begin{theorem}{\bf \emph{(Accuracy vs robustness
      trade-off)}}\label{thm:trade-off}
  Let $\mc{S}^{*}(\delta)$ denote the solution of
  \eqref{eq:Opt_prob}. Then, $\mc{S}^{*}(\delta)$ is a strictly
  decreasing function of $\delta$ in the interval
  $\delta \in [\mc{P}(\subscr{K}{kf}),\mc{P}(0)]$.
  %      % for $\delta\in[\mc{P}(\subscr{K}{kf}),\mc{P}(0)]$,
  % $\mc{S}^{*}(\delta)$ is a strictly decreasing function of $\delta$.
\end{theorem}
\begin{proof}
  From the proof of Theorem \ref{theorem: optimal gain}, we have
  \begin{equation}\label{eq: trade-off equation}
    \frac{\partial \mc S(K)}{\partial K}\Bigg|_{K^*(\lambda)} = -\lambda \frac{\partial \mc P(K)}{\partial K} \Bigg|_{K^*(\lambda)}.
  \end{equation}
  Since $\lambda> 0$ for $\delta\in[\mc{P}(K_{\text{kf}}),\mc{P}(0)]$ and
  $\mc{P}^{*}(\lambda) = \delta$, \eqref{eq: trade-off equation}
  implies that the sensitivity decreases when the error increases, and
  vice versa, so that a strict trade-off exists.
\end{proof}

\smallskip Theorem \ref{thm:trade-off} implies that there exists a
fundamental trade-off between the accuracy and robustness of a linear
filter against perturbations to measurement noise covariance
matrix. Therefore, the robustness of the linear filter in
\eqref{eq:lin_filter} in uncertain or adversarial environments can be
improved only at the expenses of its accuracy in nominal
conditions. Conversely, improving the robustness of the filter leads
to a lower accuracy in nominal conditions.

\smallskip
\begin{remark} {\bf \emph{(Design of optimally robust
      filters)}}\label{remark: robust}
  Let $\Delta R\geq 0$ denote a sufficiently small perturbation to $R$ such
  that the approximation
  $\Delta \mc{P}(K) \approx \Trace(K^{\transpose}MK\Delta R)$ holds (see
  \eqref{eq:d_dR_P}). Further, let $\Delta R$ be bounded as
  $\Trace(\Delta R)\leq \gamma$. Then, we have
  \begin{align*}
    \Delta \mc{P}(K) &= \Trace(K^{\transpose}MK \Delta R) \leq \Trace(K^{\transpose}MK) \rho(\Delta R)\\
               &= \Trace(S(K)) \rho(\Delta R) \leq \gamma \mc{S}(K).
  \end{align*}  
  Thus, given a gain $K$, the worst case performance degradation due
  to a bounded perturbation to $R$ is given by
  $\subscr{\mc{P}}{worst}(K) = \mc{P}(K) + \gamma
  \mc{S}(K)$. Therefore, a filter that is optimally robust (that is,
  it exhibits optimal worst-case performance in the presence of
  norm-bounded perturbations of the noise statistics) can be obtained
  by minimizing $\subscr{\mc{P}}{worst}(K)$. Note that this
  minimization problem is akin to the problem \eqref{eq:Opt_prob}, and
  that its solution is given by \eqref{eq:opt_gain} with
  $\lambda = \gamma^{-1}$. \oprocend
\end{remark}

\smallskip
\begin{remark} {\bf \emph{(Analysis when the system matrix $A$ is
      unstable)}}
  The accuracy-robustness trade-off shown above also holds when
  $A$ is unstable and $(A,C)$ is detectable. The analysis for this
  case follows the same reasoning as above, except that the range of
  interest for the error becomes
  $\delta \in [\mc{P}( \subscr{K}{kf}),\mc{P}(K^{*}_{\mc{S}})]$, with
  $K^{*}_{\mc{S}} = \arg \underset{K}{\min}\: \mc{S}(K).$ If $A$ does
  not have eigenvalues on the unit circle, then the Riccati equation
  \eqref{eq: riccati equation optimal gain} has a unique solution for
  $\lambda=0$ \cite{BH-AHS-TK:99} (Theorem 12.6.2), and
  $K^{*}_{\mc{S}} =K^{*}(0)$ (c.f. \eqref{eq:opt_gain}). In this case,
  $\mc{P}(K^{*}_{\mc{S}})$ is finite. The case when $A$ has
  eigenvalues on the unit circle is more involved, finding
  $K^{*}_{\mc{S}}$ is not trivial, and $\mc{P}(K^{*}_{\mc{S}})$ may
  become arbitrarily large. This aspect is left for future research
  (see Section \ref{sec: example} for an example with unit
  eigenvalues). \oprocend
\end{remark}

\smallskip
We conclude this section with an illustrative example.
% \textcolor{purple}{Next, we show the trade-off between accuracy and
%   sensitivity of the linear estimator in \eqref{eq:lin_filter} via a
%   numerical example.

\begin{example}\label{ex: trade-off}{\bf \emph{(Robustness versus performance
      trade-off)}}
  Consider the system in \eqref{eq: dynamical model} and \eqref{eq:
    output model} with matrices
  \begin{equation}
    \begin{aligned}
      A&=\begin{bmatrix}
        0.9 & 0 \\
        0.02 & 0.8
      \end{bmatrix}, \qquad C=
      \begin{bmatrix}
        0.5 & -0.8 \\
        0 & 0.7 
      \end{bmatrix}, \\
      Q&=\begin{bmatrix}
        0.5 & 0 \\
        0 & 0.7
      \end{bmatrix}, \hspace{0.9cm} R=
      \begin{bmatrix}
        0.5 & 0.1 \\
        0.1 & 0.8 
      \end{bmatrix}.
    \end{aligned}
  \end{equation}
  Fig. \ref{fig: trade-off}(a) shows the values $\mc{S}^*(\delta)$
  obtained from \eqref{eq:Opt_prob} over the range
  $\delta \in [\mc P(K_{\text{kf}}), \mc P(0)]$. Several comments are
  in order. First, as predicted by Theorem \ref{thm:trade-off}, the
  plot shows a trade-off between accuracy and robustness. Second, in
  accordance with Theorem \ref{theorem: optimal gain}, the solution to
  the minimization problem \eqref{eq:Opt_prob} implies that the
  equality constraint in \eqref{eq:Opt_prob} is active. Third, when
  $\delta=\mc{P}(\subscr{K}{\text{kf}})$, the minimization problem
  \eqref{eq:Opt_prob} returns the Kalman gain. Fourth, although the
  Kalman filter (depicted by the red dot) achieves the highest
  accuracy, it features the highest sensitivity (thus, lowest
  robustness) among the solutions of \eqref{eq:Opt_prob} over the
  range $\delta \in [\mc P(\subscr{K}{\text{kf}}), \mc P(0)]$. Thus,
  the estimator that is most accurate on the nominal data, is also the
  most sensitive to perturbations. Fifth, the linear filter obtained
  when $\delta=\mc P(0)$ exhibits the worst nominal performance, but
  is the most robust to changes in the noise
  statistics. Fig. \ref{fig: trade-off}(b) shows the values of
  $\mc{P}^*(\lambda)$ as a function of $\lambda$. We observe that
  $\mc{P}^* (\lambda)$ is a strictly decreasing function in $\lambda$
  in accordance with Corollary \ref{cor:monotonicity}. We also observe
  that the linear filter obtained when $\delta = \mc{P}(0)$, depicted
  by the green dot, has $\lambda=0$. Finally, the value
  $\mc{P}^*(\lambda)$ obtained when $\delta = \mc{P}(K_{\text{kf}})$
  cannot be shown since it requires $\lambda=\infty$.~\oprocend
\end{example}

\begin{figure}[!t]
  \centering
  \includegraphics[width=1\columnwidth,trim = 0cm 0cm 0cm 0cm, clip]{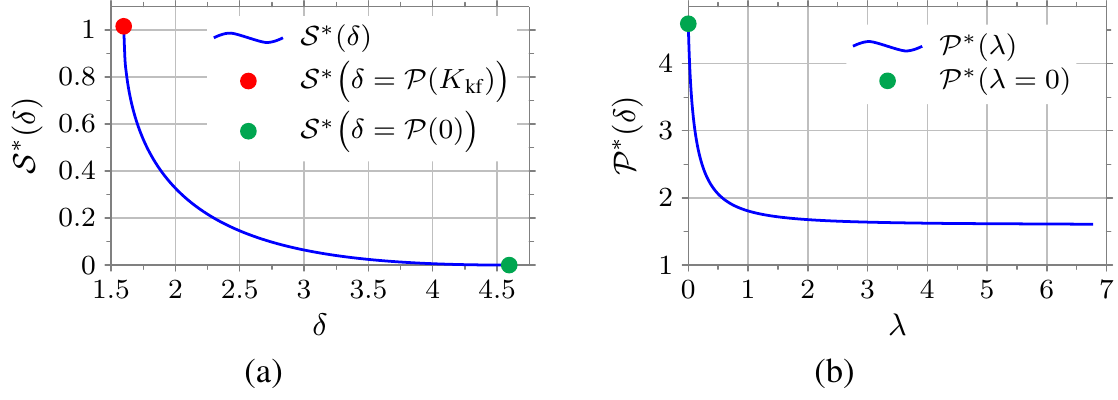}
  \caption{Panel (a) shows the accuracy versus robustness trade-off
    for the linear estimator \eqref{eq:lin_filter} and the system
    described in Example \ref{ex: trade-off}. The red dot denotes the
    Kalman filter, and the green dot denotes the linear filter with
    zero gain. The Kalman filter achieves optimal performance with the
    nominal data, yet it is the most sensitive to changes of the noise
    statistics. The opposite trade-off holds for the filter with zero
    gain. Panel (b) shows the estimation error as a function of
    $\lambda$ for the system described in Example \ref{ex: trade-off}. The green dot denotes the filter with zero gain.
    The performance of the Kalman filter does not appear in the plot
    since it requires $\lambda=\infty$.}
  \label{fig: trade-off}
\end{figure}

% \begin{remark}{\bf \emph{(Trade-off in closed loop
%       setting)}}\label{remark: trade-off closed loop}
%   We show that trade-off exists in closed loop setting, Fig. \ref{fig:
%     perc_cont}, where we consider a similar optimization problem as in
%   \eqref{eq:Opt_prob} and optimize over both the estimator and the
%   controller. In this setting, we use the cost from the linear
%   quadratic gaussian (LQG) problem to quantify the closed loop
%   performance, and use its derivative with respect to $R$ to quantify
%   its sensitivity to perturbations in the measurements. The
%   mathematical details are included in the Appendix. Fig. \ref{fig:
%     trade-off lqg} shows the trade-off curve in closed loop setting. We
%   observe that fixing one of the gains to its optimal value (i.e.,
%   Kalman gain for filter, and linear quadratic gain (LQR) gain for
%   controller) and optimize over the other does not affect the
%   trade-off. \oprocend
% \end{remark}

% \textcolor{violet}{Remark \ref{remark: trade-off closed loop} implies
% that our results for the estimator can be applied in the closed loop
% setting by fixing the controller gain to the LQR gain, without
% losing any benefit over varying both gains.}

\section{Accuracy versus robustness trade-off in perception-based
  control}\label{sec: example}
In this section we illustrate the implication of our theoretical
results to the perception-based control setting shown in
Fig. \ref{fig: perc_cont}. We consider a vehicle obeying the dynamics
\cite{SD-NM-BR-VY:19}
\begin{align}\label{eq: car's dynamics}
  x(t+1) \!= \!
  \underbrace{
  \begin{bmatrix}
    1 & T_s & 0 & 0\\
    0 & 1 & 0 & 0\\
    0 & 0& 1 & T_s\\
    0 & 0 & 0 & 1
  \end{bmatrix}}_A
                x(t) \!+\! 
                \underbrace{
                \begin{bmatrix}
                  0 & 0\\
                  T_s & 0\\
                  0 & 0\\
                  0 & T_s
                \end{bmatrix}}_B
                      u(t)
                      \!+\! w(t),
\end{align}
where $x(t) \in \mathbb{R}^4$ contains the vehicle's position and
velocity in cartesian coordinates, $u(t) \in \mathbb{R}^2$ is the
input signal, $w(t) \in \real^4$ is the process noise which follows
the same assumptions as in \eqref{eq: dynamical model}, and $T_s$ is
the sampling time. We let the vehicle be equipped with a camera, whose
images are used to extract measurements of the vehicle's position. In
particular, let
\begin{align}\label{eq: perception map}
  y(t)=f_p\big(Z(t)\big)
\end{align}
denote the measurement equation, where $y(t)\in \mathbb{R}^2$ contains
measurements of the vehicle's position,
$Z(t)\in \mathbb{R}^{p\times q}$ describes the $p\times q$ pixel
images taken by camera, and
$f_p : \mathbb{R}^{p\times q} \rightarrow \mathbb{R}^2$ is the
perception map between the camera's images and the vehicle's
position. We approximate \eqref{eq: perception map} with the following
linear measurement model (see also \cite{SD-NM-BR-VY:19}):
\begin{align}\label{eq: virtual sensor}
  y(t)=
  \underbrace{
  \begin{bmatrix}
    1 & 0 & 0 & 0\\
    0 & 0 & 1 & 0
  \end{bmatrix}}_C
                x(t) + v(t),
\end{align}
where $v(t) \in \real^2$ denotes the measurement noise, which is
assumed to follow the same assumptions as in \eqref{eq: output model}.

We consider the problem of tracking a reference trajectory using the
measurements \eqref{eq: virtual sensor} and the dynamic controller
\begin{align}\label{eq: controller}
%  \begin{split}
    \subscr{x}{c} (t+1) =& ( I - KC)(A - B L) \subscr{x}{c}(t) \nonumber \\ &+ K (y
    (t+1) - C \subscr{x}{d} (t+1) ),\nonumber \\ 
    u (t) =& - L \subscr{x}{c} (t) + \subscr{u}{d}(t),
%  \end{split}
\end{align}
where $L$ denotes the Linear-Quadratic-Regulator gain with error and
input weighing matrices $W_x > 0$ and $W_u > 0$, $K$ the gain of a
stable linear estimator as in \eqref{eq:lin_filter},\footnote{If $K$
  equals the gain of the Kalman filter for the given system, then the
  controller \eqref{eq: controller} corresponds to the
  Linear-Quadratic-Gaussian regulator.} $\subscr{x}{d}$ the desired
state trajectory, and $\subscr{u}{d}$ the control input
generating~$\subscr{x}{d}$.

% To construct an example, we consider the perception map in
% \eqref{eq: perception map} to be a convolutional neural network with
% $6$ convolutional layers with max pooling, then followed by a fully
% connected layer with dropout.\margin{Add citation to a book
% preferably} The network's loss function is the mean squared error
% between the true and the predicted labels. We collected
% $240\times 320$ pixel camera images from a vehicle simulated using
% the CARLA simulator \cite{} \margin{Add citation} labeled with the
% vehicle's true position. The data samples were collected under
% different simulated weather conditions: the data samples taken in
% clear weather are used as the nominal dataset, while the data
% samples obtained in rainy weather are used as the non-nominal
% dataset. \margin{This part is not clear to me}We trained the neural
% network offline using only the nominal dataset. Then we used the
% trained network to get position measurements for various nominal and
% non-nominal data samples, and used these measurements to estimate
% the noise statistics in \eqref{eq: virtual sensor} for each data
% sample. We denote one\margin{Why 1?}  of the nominal data samples by
% $\mathcal{D}_1$ and its corresponding measurement noise statistics
% by $v _1$. Further, we used \eqref{eq: car's dynamics} and
% \eqref{eq: virtual sensor} to design three linear filters based on
% $v _1$: the Kalman filter and two robust filters constructed using
% Remark \ref{remark: robust}. Finally, we used the non-nominal
% dataset $\mathcal{D}_2$ to obtain the non-nominal statistics of the
% noise $v_2$.
%
\begin{figure}[!t]
  \centering
  \includegraphics[width=0.9755\columnwidth,trim = 0cm 0.04cm 0cm 0cm, clip]{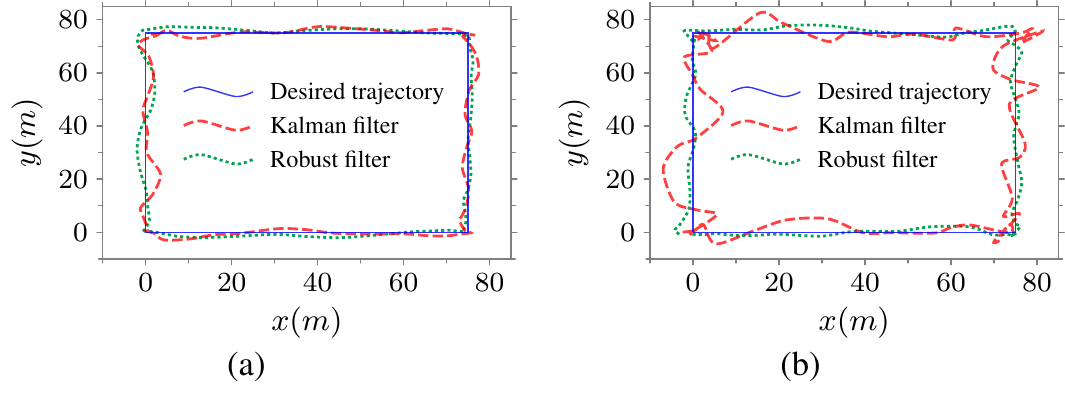}
  \caption{Panel (a) shows the trajectory tracking performance for the
    controller \eqref{eq: controller} with the Kalman filter (dashed red
    line) and a robust filter (dotted green line) in nominal noise statistics
    (the desired trajectory is shown by the solid blue line). The
    controller with the Kalman filter outperforms the other. Panel (b)
    shows the tracking performance for the two controllers using
    non-nominal noise statistics. In non-nominal conditions, the
    controller with the Kalman filter performs worse than the
    controller with the robust filter. The performance of a controller
    is measured based on the mean squared deviation between the
    controlled and nominal trajectories (see also Fig. \ref{fig:
      rmse_lqg}).}
  \label{fig: lqg}
\end{figure}
The statistics of the measurement noise in \eqref{eq: virtual sensor}
depend on how the perception map is trained and the data samples used
for the training. We aim to show that, if the estimator's gain in
\eqref{eq: controller} is designed to minimize the estimation error
based on the learned noise statistics, then the performance of the
perception-based controller \eqref{eq: controller} degrades
significantly if the learned statistics differ from the actual noise
statistics. Conversely, if the estimator's gain in \eqref{eq:
  controller} is designed based on Remark \ref{remark: robust}, then
the performance of the perception-based controller \eqref{eq:
  controller} remains robust across different values of the noise
statistics, although lower than the performance of the optimal
estimator operating with the nominal noise statistics. Fig. \ref{fig:
  lqg} shows the trajectory tracking performance for the controller
\eqref{eq: controller} for the Kalman filter and a robust filter with
$T_s = 1, Q = 0.1I_4, R=0.1 I_2, W_x =
\text{diag}(100,10^{-3},100,10^{-3}), W_u = 10^{-3}I_2$. The robust
filter corresponds to $\lambda=0.307$ (see \eqref{eq:opt_gain}). The
non-nominal covariance is $\bar{R}=2.5I_2$. We observe that the
controller based on the Kalman filter performs better in nominal
conditions, while the controller based on the robust filter performs
better in non-nominal conditions, as predicted by our theoretical
results. Fig. \ref{fig: rmse_lqg} shows the error of the Kalman filter
and the robust filter as a function of the changes of the measurement
noise covariance. We notice that for small deviations (near-nominal
conditions), the controller based on the Kalman filter performs better
than the controller based on the robust filter. However, when the
deviation of the noise statistics becomes substantially large, the
controller based on the robust filter performs better, thereby
validating our theoretical trade-off.

\begin{figure}[t]
  \centering
  \includegraphics[width=0.85\columnwidth,trim = 0cm 0cm 0cm 0cm,
  clip]{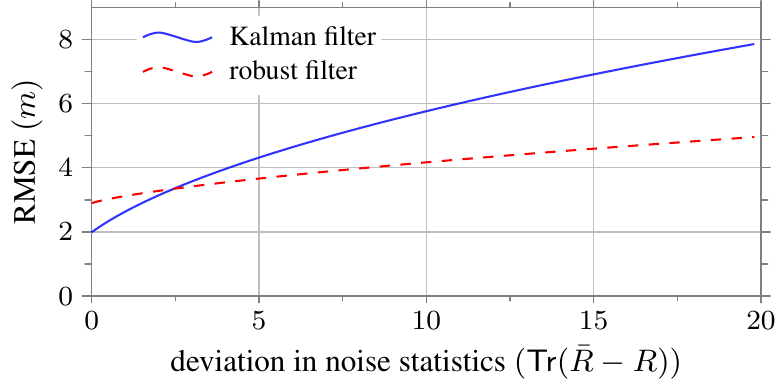}
  \caption{This figure shows the root mean square error (RMSE) of the
    controller \eqref{eq: controller} with the Kalman filter (solid blue
    line) and the robust filter (dashed red line), as a function of
    deviation between the measurement noise statistics. For small
    deviations, the controller using the Kalman filter outperforms the
    other. For large deviations, the controller using the robust
    filter outperforms the controller using the Kalman filter.}
  \label{fig: rmse_lqg}
\end{figure}

As shown in Fig. \ref{fig: perc_cont}(b), the perception error may not
be normally distributed, especially in the case of non-nominal
measurements. Although our theoretical results were obtained under the
assumption that the measurement (perception) error is normally
distributed, we next numerically show that a trade-off still exists
when the measurement (perception) error is not Gaussian. To this aim,
we consider the system in \eqref{eq: car's dynamics} and \eqref{eq:
  virtual sensor}, where the measurement noise is distributed as in
Fig. \ref{fig: perc_cont}(b) (these distributions are computed
numerically using the simulator CARLA \cite{AD-GR-FC-AL-VK:17}). We
design $6$ estimators using \eqref{eq:opt_gain} with different values
of $\delta$, and test the performance of each estimator in nominal and
non-nominal conditions. The performance of each estimator in nominal
and non-nominal environments, denoted by $\mc{P}_{\text{nom}}$ and
$\mc{P}_{\text{adv}}$, respectively, is computed using the sample
error covariance computed from the obtained samples of the estimation
error in nominal and non-nominal conditions. We approximate the
sensitivity of these estimators as the relative degradation of the
nominal performance when operating in non-nominal conditions, that is,
as $(\mc{P}_{\text{adv}} -
\mc{P}_{\text{nom}})/\mc{P}_{\text{nom}}$. %\margin{Is this correct? If so, please use this in the figure}
Fig. \ref{fig: empirical tradeoff} shows the performance and
approximate sensitivity of the estimators. It can be seen that, even
when the measurement error is not normally distributed, the estimator
with largest (respectively, smallest) accuracy also has highest
(respectively, smallest) sensitivity. These numerical results suggest
that a tradeoff exists independently of the statistical properties of
the measurement error.

\begin{figure}[t]
  \centering
  \includegraphics[width=0.85\columnwidth,trim = 0cm -0.052cm 0cm 0cm,
  clip]{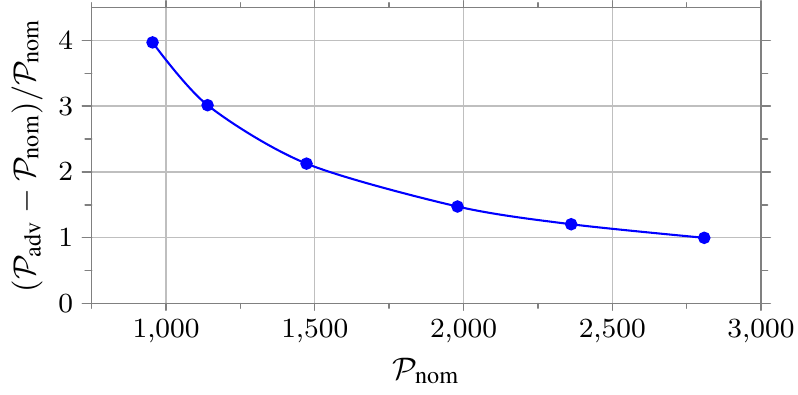}
  \caption{For the system \eqref{eq: car's dynamics} and \eqref{eq:
      virtual sensor} with measurement error distributed as in
    Fig. \ref{fig: perc_cont}(b), this figure shows the performance
    $\mc{P}_{\text{nom}}$ (i.e., trace of estimation error covariance)
    and the approximate sensitivity
    $(\mc{P}_{\text{adv}} - \mc{P}_{\text{nom}})/\mc{P}_{\text{nom}}$
    for $6$ different estimators obtained from \eqref{eq:opt_gain} by
    varying the desired accuracy $\delta$. Although the measurement
    error is not normally distributed, a trade-off still emerges
    between the accuracy of the estimators and their sensitivity.}
  \label{fig: empirical tradeoff}
\end{figure}

We conclude by showing that the identified trade-off between accuracy
and robustness of linear estimators also constrain the performance of
closed-loop perception-based control algorithms. To this aim, consider
the system \eqref{eq: car's dynamics} with controller \eqref{eq:
  controller}, where both the estimator gain $K$ and the controller
gain $L$ are now design parameters. For weighing matrices $W_x > 0$
and $W_u >0$, let the performance of \eqref{eq: controller}~be
\begin{align}\label{eq: cost function}
  \mc  J(K,L)=\mathbb{E}\Bigg[\frac{1}{T}\bigg(\sum_{t=0}^{T}x(t)^{\transpose}
  W_x x(t) + u(t)^{\transpose} W_u u(t)\bigg)\Bigg],
\end{align}
where $T$ denotes the time horizon. Notice that a lower value of the
cost $J$ is desirable, and the minimum (for $T\rightarrow \infty$) is
achieved by choosing the Kalman gain $K_{\text{kf}}$ with the linear
quadratic regulator gain $L_{\text{lqr}}$ for the matrices $W_x$ and
$W_u$. We adopt the following definition of sensitivity (this metric
is the equivalent of \eqref{eq: definition of sensitivity} for the
closed-loop performance):
\begin{align}\label{eq: sensitivity J}
  \mc S_{\mc J}(K,L)\triangleq & \Trace \Bigg[\frac{\partial \mc
                                 J(K,L)}{\partial R} \Bigg],
\end{align}
where $R$ is the noise covariance matrix of \eqref{eq: virtual
  sensor}. To see if a trade-off exists beween performance and
sensitivity of the closed-loop controller, we solve the following
problem:
\begin{equation}\label{eq:Opt_prob_lqg}
  \begin{aligned}
    \mc{S}_{\mc J}^* (\delta) = & \;\underset{K, L}{\min} \quad \mc{S}_{\mc J}(K, L)\\
    &\;\text{s.t.} \quad  \mc{J}(K, L) \leq \delta,
  \end{aligned}
\end{equation}
where $\delta$ is a constant satisfying
$\delta \geq \mc J(K_{\text{kf}} , L_{\text{lqr}})$. Notice that the
minimization problem \eqref{eq:Opt_prob_lqg} is similar to
\eqref{eq:Opt_prob} for the considered closed-loop control
setting. The results of the minimization problem
\eqref{eq:Opt_prob_lqg} are reported in Fig. \ref{fig: tradeoff lqg},
where it can be seen that a trade-off between the performance of the
controller \eqref{eq: controller} and its sensitivity still
exists. Interestingly, our numerical results show that the trade-off
curve can be obtained, equivalently, by optimizing over both the
controller and the estimator gain, by fixing the controller gain to be
the LQR gain and optimizing over the estimator gain, or by fixing the
estimator gain to be the Kalman gain and optimizing over the
controller gain. Further, if the controller gain is chosen to be the
optimal LQR gain, then the estimator gain that solves
\eqref{eq:Opt_prob_lqg} coincides with the estimator gain obtained in
Theorem \ref{theorem: optimal gain}. We leave a formal
characterization of these properties as the subject of future
investigation.

\begin{figure}[t]
  \centering
  \includegraphics[width=0.85\columnwidth,trim = 0cm 0cm 0cm 0cm,
  clip]{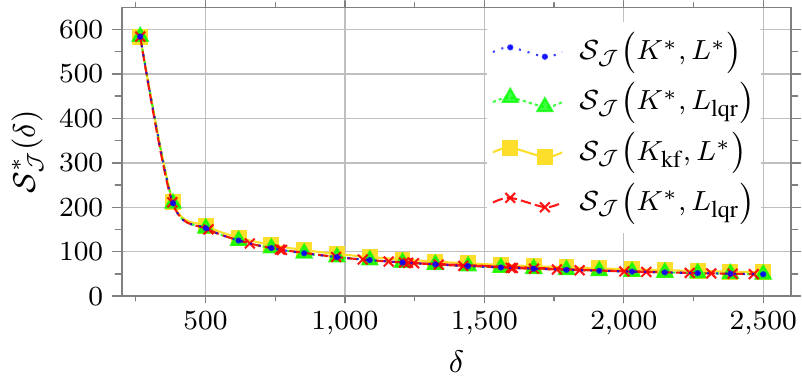}
  \caption{This figure shows the accuracy versus robustness trade-off
    in the closed loop setting described in Section \ref{sec:
      example}. The blue, green, and yellow lines denote the solution
    of \eqref{eq:Opt_prob_lqg}, where in the blue line we optimize
    over both gains, in the green line we fix the controller to the
    LQR gain and optimize over the estimator only, and in the yellow
    line we fix the estimator to the Kalman gain and optimize over the
    controller only. The red line denotes the trade-off between the
    accuracy in \eqref{eq: cost function} and the sensitivity in
    \eqref{eq: sensitivity J} with the estimator gain given in
    \eqref{eq:opt_gain} and the controller fixed to the LQR gain.}
  \label{fig: tradeoff lqg}
\end{figure}

\section{Conclusion and future work}\label{sec: conclusion}
In this paper we show that a fundamental trade-off exists between the
accuracy of linear estimation algorithms and their robustness to
unknown changes of the measurement noise statistics. Because of this
trade-off, estimators that are optimal with nominal sensing data may
perform poorly in practice due to variations of the measurements
statistics or different operational conditions. Conversely, robust
estimators obtained through a more detailed design process may
maintain similar performance levels in nominal and non-nominal
conditions, but considerably underperform in nominal conditions when
compared to nominally optimal estimators. To complement these results,
we characterize the structure of optimal estimators, for desired
levels of accuracy and robustness, and show that the trade-off also
constrain the performance of closed-loop perception-based controllers.

The results in this paper complement a recent line of research aimed
at deriving provable guarantees and performance limitations of machine
learning and data-driven algorithms
\cite{AALM-VK-FP:19,DT-SS-LE-AT-AM:19,ZD-CD-JW-YZ:19,HZ-YY-JJ-EPX-LEG-MIJ:19b},
and extend such results, for the first time, to an estimation and
control setting. This research area contains several timely and
challenging open problems, including an explicit quantification of the
performance of data-driven control algorithms when data is scarce and
corrupted, and the design of provably robust data-driven control
algorithms.

\bibliographystyle{unsrt}
\bibliography{alias,Main,FP,New}

\end{document}